\documentclass{svproc}
\usepackage{amsmath}
\usepackage{amsxtra}
\usepackage{amstext}
\usepackage{amssymb}
\usepackage{latexsym}
\usepackage{amsthm}
\usepackage{graphicx}
\usepackage{float}
\usepackage{color}
\usepackage[margin=1in]{geometry}
\usepackage{changes}
\usepackage{breqn}
\usepackage[square, numbers]{natbib}
\usepackage{thmtools}
\usepackage{colortbl}
\usepackage{setspace} 
\allowdisplaybreaks
\usepackage[justification=centering]{caption}
\usepackage[format=hang,singlelinecheck=0,font={sf,small},labelfont=bf]{subfig}
\usepackage{caption}
\usepackage[noabbrev]{cleveref}
\usepackage{enumitem}
\setlist{nolistsep,leftmargin=*}
\usepackage{mathrsfs}

\usepackage{lipsum}

\usepackage{multirow}
\usepackage[normalem]{ulem}

\usepackage{url}

\begin{document}
	
	\mainmatter              
	\title{A Note on Persistence Theory with Applications}
	\titlerunning{The Persistence Theory}  
	%
	\author{Narendra Pant\inst{*}}
	\authorrunning{Narendra Pant  }
	%
	\tocauthor{Narendra Pant}
	\institute{University of Louisiana at Lafayette\\
		Lafayette, LA 70504, USA,\\
		\vspace{0.1in}
		\email{$^*$ Email: narendra.pant1@louisiana.edu}\\ 
	}
	
	\maketitle              
	
	\begin{abstract}
	The persistence theory has been employed by several authors in order to study persistence properties of dynamical systems generated by ordinary differential equations or maps across diverse disciplines. In this note, the author discusses a roadmap to show the persitence of dynamical systems by discussing some problems from literature, which mainly involves either dealing with spectral radius or the Lyapunov exponents. 
		
	\end{abstract}
	
	\section{Introduction}
	The dynamical systems encompass a broad range of mathematical models studied across diverse disciplines, including mathematics, physics, biology, engineering, economics, and many more. A few references for different disciplines are: see \cite{guckenheimer2013nonlinear, hirsch2013differential, strogatz2018nonlinear} for examples from General Dynamical Systems, see \cite{arnol2013mathematical, ott2002chaos} for Physical Systems, \cite{kot2001elements, murray2007mathematical} for Biological and Ecological Systems, \cite{ogata2020modern, pearson2003selecting} for Engineering and Applied systems, and \cite{nowak2006evolutionary, weidlich2006sociodynamics} for Social and Economic Systems. Whichever is the discipline, the persistence results for the dynamical systems are extremely important.
	
	The persistence results are central to understanding, predicting, and controlling the dynamics of complex systems, especially when extinction or collapse has profound implications. These results bridge the gap between theoretical insights and practical applications across a range of disciplines. The persistence of predator-prey models in Ecology guarantees coexistence \cite{hofbauer1998evolutionary, murray2007mathematical}, whereas in Epidemiology, the persistence of endemic disease models helps public health experts understand when disease will continue to exist and guide strategies for eradication \cite{anderson1991infectious, brauer2012mathematical}. Similarly, the persistence of oscillatory chemical reactions in Chemistry guarantees that oscillatory behavior is maintained, which is critical for understanding natural phenomena like circadian rhythms or enzyme reactions \cite{epstein1998introduction}. The competitive market models, for example, a Cournot competition model, in Economics, the persistence of which is crucial to maintaining diversity in the market and avoiding monopolistic collapse \cite{tirole1988theory}. Further, the persistence of the magnetohydrodynamic (MHD) equations in plasma physics ensures maintaining energy production in fusion reactors \cite{freidberg2008plasma}. Also, the persistence of robotic systems in Engineering guarantees the system's robustness to disturbances or environmental changes \cite{pearson2003selecting}. These are only a few of the examples; the list is never-ending.
	
	The focus of this paper is to make a roadmap by working out a few examples from different disciplines using difference equations, differential equations, and delay differential equations. That is, given any dynamical system, how can we show that the given system is persistent?

	\section{Preliminaries}
	In this section, we summarize the persistence theory for any given dynamical system generated by ordinary differential equations or by maps. The materials are directly inspired from \cite{salceanu2009lyapunov, smith2011dynamical}. For proofs and other details, see \cite{salceanu2009lyapunov, smith2011dynamical}. To this end, we consider following dynamical systems

	\begin{enumerate}[label =\alph*). ]
		\item \textbf{Discrete-time Systems ($t \in \mathbb{Z}_+$)}
		
		Consider the discrete dynamical system 
		\begin{align}
			\label{thoery model 1}
			\begin{split}
				& z(t+1) = F(z(t))  \\
				& z_0 = z(0) \in \mathbb{R}^p_+ \times \mathbb{R}^q_+
			\end{split}
		\end{align}
			To study the persistence of $y$, we put the above system in the following form
		\begin{align} 
			\label{discrete time}
			\begin{split}
				&	x(t+1) = \mathcal{F}(z(t)) \\
				& y(t+1) = \mathcal{A}(z(t)) y(t)
			\end{split}
		\end{align}

		\item \textbf{Continuous-time Systems ($t \in \mathbb{R}$)}
		
		Consider the continuous dynamical system 
		\begin{align}
			\label{thoery model 2}
			\begin{split}
				& z'(t) = F(z(t))  \\
				& z_0 = z(0) \in \mathbb{R}^p_+ \times \mathbb{R}^q_+
			\end{split}
		\end{align}
		To study the persistence of $y$, we put the above system in the following form
			\begin{align} 
			\label{continuous time}
			\begin{split}
				&	x'(t) = \mathcal{F}(z(t)) \\
				& y'(t) = \mathcal{A}(z(t)) y(t)
			\end{split}
		\end{align}
		
	\end{enumerate}
	Some of the common assumptions are: Let $z = (x,y) \in \mathbb{R}^n$, where $x \in \mathbb{R}^p$, $y \in \mathbb{R}^q$, such that $n = p+q$, and $z \mapsto ( \mathcal{F}(z), \mathcal{A}(z) )$ is a continuous map. Here $F: \mathbb{R}^p_+ \times \mathbb{R}^q_+ \rightarrow \mathbb{R}^p_+ \times \mathbb{R}^q_+$ is a continuous map, $\mathcal{F}: \mathbb{R}^p_+ \times \mathbb{R}^q_+ \rightarrow \mathbb{R}^p_+$, $g: \mathbb{R}^p_+ \times \mathbb{R}^q_+ \rightarrow \mathbb{R}^q_+$, such that 
	\[
	F(z) = (\mathcal{F}(z), g(z)) \quad \text{with} \quad g(z) = \mathcal{A}(z)y, \ \forall z \in \mathbb{R}^p_+ \times \mathbb{R}^q_+.
	\]
	Further, $\mathcal{A}(z)$ is a continuous matrix function with $\mathcal{A}(z) \in \mathbb{R}^{q \times q}$ satisfying $0 \le \mathcal{A}(x,0)$. The existence and uniqueness of the solutions for all times are assumed for both discrete and continuous cases. Let $\Phi(t,z) = z(t) = (x(t), y(t))$ be the (state) continuous semiflow generated by the solutions of (\ref{discrete time}) or (\ref{continuous time}). 
	
	Now, we list the main steps to show the persistence of $y$ in the case of discrete-time systems. Since the setting of the working problem is similar, the steps will also be valid for continuous-time systems (of course, the basics of ordinary differential equations will apply for the continuous case).

	\begin{enumerate}[label =\alph*). ]
		\item Rewrite the system (\ref{thoery model 1}) as in (\ref{discrete time}).
		
		\item Define the persistent function $\rho(z)$ as per requirement. For example, persistence of which species do we want to study? Of a species, of at least one of the species, or of all the species? A few examples of persistent functions are $\rho(z) = \sum_{i} |y_i|$, $\rho(z) = y$, etc.
		
		\item Once the persistent function is defined, we define the extinction set as 
		$$
		X_0^y := \{ z(t, z_0) \mid \rho(\Phi(t,z_0)) = 0, \ \forall t \ge 0\} \equiv \{ (x_0, y_0) \mid y_0 = 0\}.
		$$
		
		\item Study the dynamics on $X_0^y$, on which the subsystem is given by
		\begin{align}
			\begin{split}
				& \text{Discrete: } x(t+1) = \mathcal{F}(x(t),0), \\
				& \text{Continuous: } x'(t) =  \mathcal{F}(x(t),0).
			\end{split}
		\end{align}
		Find the $\omega$-limit set of $X_0^y$, that is, $\Omega (X_0^y):= \{ M_1, M_2, \cdots \}$.
		
		\item Write the linearized systems as
		\begin{align}
			\label{linearized system}
			\begin{split}
				& \text{Discrete: } u(t+1) = \mathcal{A}(\Phi(t, z)) u(t), \hspace{1cm} u(0) = y(0), \\
				& \text{Continuous: } v'(t) = \mathcal{A}(\Phi(t, z)) v(t), \hspace{1cm} v(0) = y(0).
			\end{split}
		\end{align}
		Let the $q \times q$ fundamental solution of \eqref{linearized system} be $P(t,z)$. Then, one can write
		\begin{align}
			\begin{split}
				& \text{Discrete: } 	P(t+1,z) =  \mathcal{A}(\Phi(t, z)) P(t,z), \\
				& \hspace{1cm} \equiv P(t,z) = \mathcal{A}(z(t-1)) \mathcal{A}(z(t-2)) \cdots  \mathcal{A}(z(1))  \mathcal{A}(z(0)), \\
				& \text{Continuous: } \frac{d}{dt} P(t,z) = \mathcal{A}(\Phi(t, z)) P(t,z),
			\end{split}
		\end{align}
		with $P(0,z) = I$ in either case.
		
		\item To show that $\Omega (X_0^y)$ is a uniformly weak repeller set, we have the following two cases:
		\begin{enumerate}[label =\roman*). ]
			\item When some of the $M_i \in \Omega (X_0^y)$ are equilibrium points and the remaining are intervals or sets.  
			
			In this case, use Proposition 1 from \cite{salceanu2009lyapunov}. It uses the concept of Lyapunov exponents.
			
			\item When all of the $M_i \in \Omega (X_0^y)$ are equilibrium points.  
			
			In this case, use Corollary 1 from \cite{salceanu2009lyapunov}. It uses the concept of finding the spectral radius of periodic orbits from $\Omega (X_0^y)$.  
			
			Either of these techniques can still be used even when the scenarios are the other way around.
		\end{enumerate}
		
		\item After showing that $\Omega (X_0^y)$ is a uniformly weak repeller set, use Theorem 2.3 from \cite{salceanu2009lyapunov}. By this theorem, there exists $\epsilon > 0$ such that $\liminf_{t \rightarrow \infty } | y(t) | > \epsilon$ for all positive initial conditions $x(0), y(0)$. Hence, by definition, $y$ is uniformly strongly $\rho$-persistent.
		
	\end{enumerate}

	\section{Main Results: Problem Solving }

	\subsection{A discrete-time predator-prey model from Ackleh et al. \cite{ackleh2023discrete}}
		
	Consider a discrete-time predator-prey periodic model in \cite{ackleh2023discrete} by Ackleh et al. (2024). The model equations are listed below as:
		\begin{align}
			\label{periodic problem}
		\begin{split}
			& n(t+1) = \phi(t, n(t))  \hspace{0.1cm}[1-f(p(t))p(t)] \hspace{0.1cm} n(t)\\
			& p(t+1) = s_pp(t) + \kappa  \hspace{0.1cm} \phi (t, n(t))n(t) f(p(t))p(t).
		\end{split}
	\end{align}
	where $n(t)$ and $p(t)$ are the densities of the prey and predator population, respectively, at time $t$. In the absence of the predator, the prey is assumed to grow with the time-dependent nonlinear growth rate $ \phi(n(t), t)$. We assume that  $0<s_n<1$ is the (density-independent) survival probability of prey and $b(n,t)$ is the time and density-dependent prey fecundity.  To represent seasonal breeding, we choose $b(n,t)$ to be a periodic function of period two of the form $b(n,2t)=\hat{b}(n), \; b(n, 2t+1)=0$, $t=0,1,2,\dots$. Thus, the seasonal growth rate is given by $\phi(n,2t) = s_n +\hat b(n)$ and $\phi(n,2t+1) = s_n$, $t=0, 1, 2, \dots$. This form implies that individuals in the prey population reproduce during even time steps (breeding seasons) but not during odd time steps (non-breeding seasons). 
	
	The prey growth rate is modified by the factor $1-f(p)p$, where $0 < f(p) < 1$ is defined to be the probability that an individual prey is consumed by an individual predator when $p$ predators are present. Therefore, $ 0 \le  f(p)p < 1$ gives the probability that an individual prey is consumed when $p$ predators are present, and the term $0\le 1-f(p)p <1$ represents the fraction of prey that manage to escape predation. We assume that predator reproduction is proportional to the amount of prey consumed, where $\kappa>0$ converts consumed prey into new predator individuals. In addition, we assume that the predator has a density-independent survival probability $0 < s_p < 1$.  For more details, we refer the reader to \cite{ackleh2023discrete}.
		
	In order to discuss the periodic solutions of above model, we first time-tansform it using realtioin $\tau +i = 2(t+i), i = 0, 1, 2 \cdots$. Then, the composite system takes the form
	\begin{align}
		\label{Composed Seasonal Breeding with compact symbols}
		\begin{split}
			&n(\tau +1)= AB \left[1-f(x)x \right]  n(\tau),\\
			&p(\tau +1)=  s_px+ \kappa n(\tau) A B f(x)x,
		\end{split}
	\end{align}
	where
	\begin{align}\label{notatioin ABx}
		\begin{split}
			& A = A(n)= s_n^2+s_n\hat{ b}(n), \\
			& B = B(p)=1-f(p)p , \\
			& x = x(n, p)=s_p p+ \kappa \left(s_n+\hat{ b}(n)\right)  n f(p)p \equiv s_pp + \frac{\kappa}{s_n}A(n) nf(p) p. \\
		\end{split}
	\end{align}
		In Theorem \ref{Permanence}, we show that system \eqref{Composed Seasonal Breeding with compact symbols} is persistent when these two growth rates are greater than one. This persistence, combined with solution boundedness obtained in the proof of Lemma 2 in \cite{ackleh2023discrete}, establishes the permanence of the system. To this end, the growth rates of the species related to dynamics of the above model are
			\begin{enumerate}[label =\alph*). ]
					\item The inherent growth rate of prey is $r_0:= \phi(0) = s_n^2 + s_n\hat{ b}(0)$.
			
				\item The invasion growth rate of predator is $r_i:= s_p + \kappa \bar{n} f(0)$ where $\bar{n} = \hat{ b}^{-1} \left( \frac{1-s_n^2}{s_n}\right)$.
				\end{enumerate}
	\begin{theorem} 
		\label{Permanence}
		Assume that $r_0 >1$ and  $r_i>1$. Then system \eqref{Composed Seasonal Breeding with compact symbols} is strongly uniformly persistent, i.e., there exists an $\epsilon >0$ such that $\liminf_{\tau \to \infty} \min\{n(\tau), p(\tau)\} > \epsilon$ for all positive initial conditions $(n(0), p(0))$. 
	\end{theorem} 
	\begin{proof}
		We first rewrite the system (\ref{Composed Seasonal Breeding with compact symbols}) in the form given in \cite{salceanu2009lyapunov} as follows: 
		\begin{align*}
			\begin{split}
				&n(\tau +1)=\mathcal{A}(n(\tau ), p(\tau ))n(\tau ),\\
				&p(\tau +1)=\mathcal{F}(n(\tau ), p(\tau )),
			\end{split}
		\end{align*}\\
		where $\mathcal{A}(n(\tau), p(\tau )) := AB [1-f(x)x]$ and $\mathcal{F}(n(\tau ), p(\tau )) := s_px+ \kappa n A B f(x)x$. Consider the boundary dynamics where $n=0$, which is given by $p(\tau +1) = s_p^2 p(\tau ) $.  Since $0<s_p <1$, it follows that  $\lim_{{\tau \to \infty}} p(\tau) =0$. Note that $(0,0)$ is unstable when $r_0>1$. Define
		\begin{align*}
			&Z:= \{(n,p) \in  \mathbb{R}^2 | n \ge 0, p \ge 0 \},\\
			&X_1:=\{ (n,p) \in  Z | n=0 \},
		\end{align*}
		and $M\subseteq Z \cap X_1$. Take $M:= \{(n,p)\in  \mathbb{R}_+^2 | n=0, 0 \le p \le \widetilde D\}$	where $\widetilde D$ is the bound from Lemma 2 in \cite{ackleh2023discrete}.  Then $M$ is non-empty, compact, and positively invariant. Since $\lim_{{\tau \to \infty}} p(\tau) =0$, the omega limit set of $M$ is 
		$$
		\Omega (M)= \cup _{z \in M} \omega (z) = \{(0,0)\}.
		$$
		Now let $\mathcal{P}:= \{E_0\}=\{(0,0)\}.$ Then $\mathcal{A}(E_0)=A(0) \neq 0$, and hence $\mathcal{A}(E_0)$ is primitive with $r(\mathcal{P})=r_0 >1$. Take the unit vector $\eta =1$, and $(n,p) \in M,$ then $\mathcal{A}(n,p)\eta \ne 0, \forall (n,p) \in M$. So, all the three conditions of Corollary 1 of \cite{salceanu2009lyapunov} are satisfied. Hence, by Corollary 1 and Proposition 1 in \cite{salceanu2009lyapunov}, $M$ is a uniformly weak repeller set relative to the dynamics on $Z \backslash X_1$. Finally by Theorem 2.3 in \cite{salceanu2009lyapunov}, there exists an $ \widehat \epsilon >0$ such that $\liminf_{\tau \to \infty} |n(\tau )| > \widehat  \epsilon$ for all $(n(0), p(0)) \in Z \backslash X_1.$ Thus, by definition, $n$ is strongly uniformly persistent.
		
		Now, we again rewrite the system (\ref{Composed Seasonal Breeding with compact symbols}) as follows: 
		\begin{align*}
			\begin{split}
				&n(\tau +1)= \mathcal{F}(n(\tau ), p(\tau )),\\
				&p(\tau +1)=\mathcal{A}(n(\tau ), p(\tau ))p(\tau ),
			\end{split}
		\end{align*}
		where $\mathcal{F}(n(\tau ), p(\tau )):= AB [1-f(x)x] n(\tau )$ and 
		$$\mathcal{A}(n(\tau ), p(\tau )):= s_p \left( s_p+ \frac{\kappa}{s_n} A(n(\tau ))n(\tau ) f(p(\tau )) \right) + \kappa  A(n(\tau ))n(\tau ) B(p(\tau )) f(x) \left( s_p+ \frac{\kappa}{s_n} A(n(\tau ))n(\tau ) f(p(\tau )) \right). $$		
		Consider the boundary dynamics where $p=0$, which is given by $ n(\tau +1) = \left( s_n^2+s_n\hat{ b}(n(\tau )) \right)n(\tau )$. Note that from Lemma 1 in \cite{ackleh2023discrete},  $\lim_{{\tau \to \infty}} n(\tau) = \bar{n}$, and that  $(\bar{n},0)$ is unstable when $r_i> 1$. Define
		\begin{align*}
			\begin{split}
				&Z:= \{(n,p) \in  \mathbb{R}^2 | n \ge 0, p \ge 0 \},\\
				& X_1:=\{(n,p) \in  Z |p=0\},
			\end{split}
		\end{align*}
		and $M\subseteq (Z \cap X_1)$. Take $M:= \{(n,p)\in  \mathbb{R}_+^2 | \frac{\epsilon} {2}  \le n\le D, p=0 \}$	where $D$ is the bound from Lemma 2 in \cite{ackleh2023discrete}. Then $M$ is non-empty, compact, and positively invariant. Since $\lim_{{\tau \to \infty}} n(\tau) =\bar{n}$, the omega limit set of $M$ is 
		$$
		\Omega (M)= \cup _{z \in M} \omega (z) = \{(\bar{n},0)\}.
		$$
		Now let $\mathcal{P}:= \{E_1\}=\{(\bar{n},0) \}$.  Then $\mathcal{A}(E_1)= r_i \neq 0$, and hence $\mathcal{A}(E_1)$ is primitive with $r(\mathcal{P})=r_i>1$. Take the unit vector $\eta =1,$ and $(n,p) \in M,$ then $\mathcal{A}(n,p)\eta \ne 0, \forall (n,p) \in M$.	So, by using a similar argument as above, we can show that  there exists $\widetilde\epsilon>0$ such that $\liminf_{\tau \to \infty} |p(\tau )| > \widetilde\epsilon$ for all $(n(0), p(0)) \in Z \backslash X_1.$ Thus, by definition, $p$ is also strongly uniformly persistent.
		
		Finally, let $\epsilon = \min\{\widehat \epsilon, \widetilde \epsilon \} $, then  the result follows and in particular, $\liminf_{\tau \to \infty} \min\{n(\tau), p(\tau)\} > \epsilon$ for all positive initial conditions $(n(0), p(0))$. 
	\end{proof}

	\subsection{ A discrete-time predator-prey model from Q. Din \cite{din2017complexity}}
	Consider a discrete-time predator-prey model in \cite{din2017complexity}. The model equations are listed below as:
		\begin{align}
		\begin{split}
			\label{Predator-prey model}
			& x(t+1) =  x(t) e^{\left[r\left(1-\frac{x(t)}{K}\right)-\frac{\beta y(t)}{x(t)+\gamma}\right]} \\
			& y(t+1) = y(t)e^{\left[1- d- \frac{a y(t)}{bx(t)+c}\right]}
		\end{split}
	\end{align}
	where $x(t)$ and $y(t)$ are prey and predator populations at time $t$ respectively. Furthermore, the parameter $r$ denotes the intrinsic growth rate of prey, $\beta$ represents the maximum value of the per capita reduction rate of prey, $K$ denotes environmental carrying capacity of the prey in a particular habitat, $\gamma$ measures the extent to which the environment provides protection to prey, $c$ measures the extent to which the environment provides protection to predator, and $a$ represents the maximum value of the per capita reduction rate of predator. Assume that $b$ measure the food quality that the prey provides for conversion into predator births, and $d$ denotes the death rate of predator. For more details, we refer reader to \cite{din2017complexity}
	
	We discuss the persistence of system \eqref{Predator-prey model} in the Theorem 2. To this end, first we have the following Lemma 1, where we discuss the persistence of the prey population. In Lemma 2, the persistence of the predator population will be discussed.
	\begin{lemma}
		\label{persistence of prey}
		Assume $r \gamma a > \beta c(1-d)$. Then, the prey population is uniformly strongly persistent, that is, there exists $\epsilon_1 >0$ such that $\liminf _{t \rightarrow \infty } | x(t) | > \epsilon_1$ for all positive initial conditions $x(0), y(0)$.
	\end{lemma}
	\begin{proof}
			First, we rewrite the system (\ref{Predator-prey model}) as
		\begin{align*}
			\begin{split}
				& x(t+1) = A(x(t), y(t))x(t), \\
				& y(t+1) = f(x(t), y(t)), \hspace{0.5cm} \text{where } A(x,y) = e^{\left[r\left(1-\frac{x}{K}\right)-\frac{\beta y}{x+\gamma}\right]}, \quad 
				f(x,y) = y e^{\left[1- d- \frac{a y}{bx+c}\right]}. \\
			\end{split}
		\end{align*}
		The extinction set is
		\begin{align*}
			\begin{split}
				X_0 &= \{ z(t,z_0) \in Z \mid x(t)=0, \ \forall t \ge 0\}  \equiv \{ (x_0, y_0) \in \mathbb{R}^2_+ \mid x_0=0\}.
			\end{split}
		\end{align*}
		In what follows, we study the dynamics on $X_0$, which is given by 
		\[
		y(t+1) = y(t) e^{1-d-\frac{a y(t)}{c}}.
		\]
		Clearly, the fixed points are $y=0$ and $\bar{y} = (1-d)\frac{c}{a}$. Observe that if we let $f_2(y) = y e^{1-d-\frac{a y}{c}}$, then
		\[
		f_2'(y)= e^{1-d-\frac{a y}{c}}\left[1-\frac{a y}{c}\right].
		\]
		Further, for $y > c/a$, $f_2'(y)<0$, and for $0<y<c/a$, $f_2'(y)>0$. Therefore, the function $f_2(y)$ is increasing on $\left[0, \frac{c}{a}\right]$ and decreasing on $\left[\frac{c}{a}, \infty\right)$. Let $\bar{y} = (1-d)\frac{c}{a}$, and consider the following cases:
		\begin{enumerate}[label=\alph*). ]
			\item When $y_0 = 0$, then $y(1)=0$.
			
			\item When $0 < y_0  < \bar{y}$, then $0 < f_2(y_0)< f_2(\bar{y}) \implies 0< y(1)< \bar{y}$.
			
			\item When $\bar{y} < y_0  < \frac{c}{a}$, then $f_2(\bar{y}) < f_2(y_0)< f_2\left(\frac{c}{a}\right) \implies \bar{y}< y(1)< f_2\left(\frac{c}{a}\right)$.
			
			Note from a), b), c) that the maximum value of iterations of $f_2$ is $f_2\left(\frac{c}{a}\right)$.
			
			\item  For any $y_0 \in \left(\frac{c}{a}, \infty\right)$, the next iteration goes to the left of $y_0$ since $f_2$ is decreasing in this region. When $y_0 > \frac{c}{a}$, $f_2(y_0)<f_2\left(\frac{c}{a}\right) \implies y(1) < f_2\left(\frac{c}{a}\right)$. Note that here $y(1)$ lies to the right of $\frac{c}{a}$. Again, since the function is decreasing in this region, the minimum it can decrease to is $f_2\left(f_2\left(\frac{c}{a}\right)\right) = f_2^2\left(\frac{c}{a}\right)$. 
		\end{enumerate}
		Therefore, one can write
		\[
		f_2^2\left(\tfrac{c}{a}\right) < (1-d)\tfrac{c}{a}<\tfrac{c}{a}<f_2\left(\tfrac{c}{a}\right).
		\]
		So, from the above discussion, $M_1 = \{(0,0)\}$ and $M_2 = \{ (0,M)\}$ where $M =  \left[ f_2^2\left(\tfrac{c}{a}\right), f_2\left(\tfrac{c}{a}\right) \right]$. Therefore,
		\[
		\Omega (X_0) = \{ M_1, M_2\}.
		\]
		
		Now, consider the linearized system 
		\[
		u(t+1)= A(x(t), y(t)) u(t), \quad A(x,y) = e^{ \left[r\left(1-\frac{x}{K}\right) - \frac{\beta y}{x+\gamma}\right]}.
		\]
		The Lyapunov exponent is given by
		\begin{align}
			\begin{split}
				\label{Lyapunov exponent1}
				& \lambda (z, \eta) = \limsup_{t \rightarrow \infty } \frac{1}{t} \ln |P(t,z)\eta|, \\
				& \lambda (z) =  \limsup_{t \rightarrow \infty } \frac{1}{t} \ln |P(t,z)|, \quad \text{for unit vector } \eta = 1.
			\end{split}
		\end{align}
		We find $P(t,z)$ as follows:
		\begin{align*}
			\begin{split}
				P(t,z) &= A(x(t-1), y(t-1)) \cdot A(x(t-2), y(t-2)) \cdot \dots  \cdot A(x(1), y(1)) \cdot A(x_0, y_0) \\
				&= \exp\left[r \left(1- \frac{x(t-1,z)}{K}\right) - \frac{\beta y(t-1,z)}{x(t-1,z)+\gamma}\right] \\
				& \quad \cdot \exp\left[r \left(1- \frac{x(t-2,z)}{K}\right) - \frac{\beta y(t-2,z)}{x(t-2,z)+\gamma}\right] \dots \\
				& \quad \cdot \exp\left[r \left(1- \frac{x(0,z)}{K}\right) - \frac{\beta y(0,z)}{x(0,z)+\gamma}\right] \\
				&\\
				&= \exp \left[rt -\frac{S_1(t,z)}{K} - \beta S_2(t,z)\right].
			\end{split}
		\end{align*}
		Taking natural logs on both sides,
		\begin{align}
			\label{log of P}
			\ln P(t,z) = rt -\frac{S_1(t,z)}{K} - \beta S_2(t,z),
		\end{align}
		where 
		\begin{align*}
			\begin{split}
				&S_1(t,z) = x(t-1,z)+ x(t-2,z)+ \dots + x(0,z), \\
				&S_2(t,z) = \frac{y(t-1,z)}{x(t-1,z)+\gamma} + \frac{y(t-2,z)}{x(t-2,z)+\gamma} + \dots + \frac{y(0,z)}{x(0,z)+\gamma}.
			\end{split}
		\end{align*}
		Consider the following cases:
		\begin{enumerate}[label=\alph*). ]
			\item Take $z_0= (x_0,y_0) = (0,0)$. Clearly, $S_1(t, z_0) = 0$ as $x_0=0$, and $S_2(t, z_0) = 0$ as $y_0 =0$. Therefore, equation (\ref{log of P}) takes the form $\ln |P(t, z_0)| = rt$. Finally, from equation (\ref{Lyapunov exponent1}), $\lambda (z_0) =  \limsup_{t \rightarrow \infty } \tfrac{1}{t} (rt) = r >0$.

			\item Take $z_0 = (x_0,y_0)$ where $x_0=0$ and $y_0 \in M$. As above, $S_1(t,z_0) =0$, and the expression for $S_2(t,z_0)$ takes the form 
			\[
			\gamma S_2(t,z_0) = y(t-1,z_0) + y(t-2, z_0) + \dots + y(0,z_0).
			\]
			To compute $S_2(t,z_0)$, consider
			\begin{align*}
				\begin{split}
					y(t) &= y(0) \exp\left[t- dt -\tfrac{a}{c} \{ y(t-1) + y(t-2)+ \dots + y(0)\}\right].
				\end{split}
			\end{align*}
			From this, one can get
			\begin{align*}
				\ln \left|\tfrac{y(t)}{y(0)}\right| &=  t- dt -\tfrac{a}{c} \{ y(t-1) + \dots + y(0)\} \implies 
				y(t-1) + \dots + y(0) = \tfrac{c}{a} \left[ t- dt - \ln  \left|\tfrac{y(t)}{y(0)}\right|  \right].
			\end{align*}
			At $z=z_0$, 
			\begin{align*}
				S_2(t,z_0) &= \tfrac{c}{\gamma a} \left[ t- dt - \ln  \left|\tfrac{y(t)}{y(0)}\right|  \right].
			\end{align*}
			Using the values of $S_1(t,z_0)$ and $S_2(t,z_0)$ in equation (\ref{log of P}), we get 
			\begin{align*}
				\ln |P(t,z_0)| &= rt-\tfrac{\beta c}{\gamma a} \left[t- dt -  \ln  \left|\tfrac{y(t)}{y(0)}\right|  \right], \\
				\lambda (z_0) &=  \limsup_{t \rightarrow \infty } \tfrac{1}{t} \left[ rt-\tfrac{\beta c}{\gamma a} \left(t- dt -  \ln  \left|\tfrac{y(t)}{y(0)}\right|  \right) \right] \\
				&= r - \tfrac{\beta c}{\gamma a} (1-d) - \limsup_{t \rightarrow \infty} \tfrac{1}{t} \ln \left|\tfrac{y(t)}{y(0)}\right| \\
				&= r - \tfrac{\beta c}{\gamma a} (1-d) \quad \text{(since $y(t)$ is bounded)} \\
				&> 0 \quad \text{if } r  >  \tfrac{\beta c}{\gamma a} (1-d) \equiv  r \gamma a > \beta c(1-d).
			\end{align*}
		\end{enumerate}
		
		Since all the Lyapunov exponents are positive, by Proposition 1 of \cite{smith2011dynamical}, all the $M_i$ in $\Omega (X_0)$ are uniformly weak repellers. By Theorem 2.3 of the same paper \cite{smith2011dynamical}, we may choose $\epsilon_1>0$ such that $\liminf_{t \rightarrow \infty } |x(t)| > \epsilon_1$, for all $y_0 \in Z\backslash X_0$. By definition, $y$ is uniformly strongly persistent. This completes the proof of the theorem.
	\end{proof}
	
Next, in the following lemma, we discuss the persistence of the predator population.

	\begin{lemma}
	\label{persistence of predator}
	Suppose $ b < 1$. Then, the predator is uniformly strongly persistent, that is, there exists $\epsilon_2 >0$ such that $\liminf _{t \rightarrow \infty } | y(t) | > \epsilon_2$ for all positive initial conditions $x(0), y(0)$.
	\end{lemma}
	\begin{proof}
	The extinction set 
	\begin{align*}
		\begin{split}
			X_1 &= \{ z(t,z_0) \in Z \mid y(t)=0, \ \forall t \ge 0\}, \hspace{0.5cm} \text{where } z(t) = (x(t), y(t)) \\
			& \equiv \{ (x_0, y_0) \in \mathbb{R}^2_+ \mid y_0=0\}.
		\end{split}
	\end{align*}
	Dynamics on the extinction set is given by
	\[
	x(t+1) = x(t) e^{r\left(1-\frac{x(t)}{K}\right)}.
	\]
	Clearly, the fixed points are $\bar{x}=0$ and $\bar{x}=K$. So, with a similar argument as for the prey, we get $M_1 = \{(0,0)\}$ and $M_2 = \{ (M,0)\}$ where 
	\[
	M =  \left[ f_1^2\!\left(\tfrac{K}{r}\right), f_1\!\left(\tfrac{K}{r}\right) \right].
	\]
	Therefore,
	\[
	\Omega (X_1) = \{ M_1, M_2\}.
	\]
	
	Now, consider the linearized system 
	\[
	u(t+1)= A(x(t), y(t)) u(t), \hspace{0.5cm} \text{where } A(x,y) = e^{r\left(1-d - \frac{a y}{b x+ c}\right)}.
	\]
	The Lyapunov exponent is given by
	\begin{align*}
		\begin{split}
			\lambda (z, \eta) &= \limsup_{t \rightarrow \infty } \frac{1}{t} \ln |P(t,z)\eta|, \\
			\lambda (z) &=  \limsup_{t \rightarrow \infty } \frac{1}{t} \ln |P(t,z)|, \quad \text{for unit vector } \eta = 1.
		\end{split}
	\end{align*}
	For $z_0 \in M_1$, we have 
	\begin{align*}
		\begin{split}
			\lambda (z_0) &= \limsup_{t \rightarrow \infty } \frac{1}{t} \ln |P(t,z)| \\
			&= \limsup_{t \rightarrow \infty } \frac{1}{t} \ln \big|e^{t(1-d)}\big| \\
			&= \limsup_{t \rightarrow \infty } \frac{1}{t} \, t(1-d) \\
			&= (1-d) > 0 \quad \implies d<1.
		\end{split}
	\end{align*}
	For $z_0 \in M_2$, we have 
	\begin{align*}
		\begin{split}
			\lambda (z_0) &= \limsup_{t \rightarrow \infty } \frac{1}{t} \ln |P(t,z)| \\
			&= \limsup_{t \rightarrow \infty } \frac{1}{t} \ln \big|e^{t(1-d)}\big| \\
			&= \limsup_{t \rightarrow \infty } \frac{1}{t} \, t(1-d) \\
			&= (1-d) > 0 \quad \implies d<1.
		\end{split}
	\end{align*}
	Since all the Lyapunov exponents are positive, by Proposition 1 of \cite{smith2011dynamical}, all the $M_i$ in $\Omega (X_1)$ are uniformly weak repellers. By Theorem 2.3 of \cite{smith2011dynamical}, we may choose $\epsilon_2>0$ such that 
	\[
	\liminf_{t \rightarrow \infty } |y(t)| > \epsilon_2, \quad \forall y_0 \in Z\backslash X_1.
	\]
	By definition, $x$ is uniformly strongly persistent.
	\end{proof}
	
	Combining the Lemmas  \ref{persistence of prey} and \ref{persistence of predator}, we have the following theorem:
	
	\begin{theorem}
		\label{theorem Q din}
	Suppose $d<1$ and $r\gamma a>\beta c(1-d)$. Then the system is uniformly persistent if there exist a $\epsilon= \min \{\epsilon_1, \epsilon_2\}$ such that $\liminf_{t \to \infty} \min \{x(t),y(t)\}>\epsilon$, $\forall$ $x_0, y_0>0$.
	\end{theorem}
	\begin{proof}
		The proof directly follows from the Lemmas \ref{persistence of prey} and \ref{persistence of predator}.
	\end{proof}

			\subsection{ A three species food chain model from \cite{alebraheem2012persistence}}
			Consider the system in \cite{alebraheem2012persistence} where two predators are competing on a prey. Since our focus is to study persistence, we take nondimensional form of the original system. For more details, we refer the reader to \cite{alebraheem2012persistence}.
			\begin{align}
				\label{two predator model}
				\begin{split}
				&	\frac{dx}{dt} = x(1-x) - \frac{\alpha xy}{1+h_1\alpha x} - \frac{\beta xz}{1+h_2\beta x} \\
				&  \frac{dy}{dt} =-uy + \frac{e_1 \alpha xy}{1+h_1\alpha x} - \frac{e_1 \alpha y^2}{1+h_1 \alpha x} -c_1yz \\
				&  \frac{dz}{dt} =-wz + \frac{e_2 \alpha xz}{1+h_2 \beta x} - \frac{e_2 \beta z^2}{1+h_2\beta  x} -c_2yz \\
				\end{split}
			\end{align}
		The above equations in system \eqref{two predator model} are of Kolgomorov type. Following are a few assumptions made originally in \cite{alebraheem2012persistence}:
		\begin{enumerate}[label=\alph*).]
			\item $x$ is a prey population and $y, z$ are competing predators, living exclusively on the prey, i.e. 
			$$
			\frac{\partial L}{\partial y_i}<0, \frac{\partial M_i}{\partial x}>0, M_i(0,y,x) <0, \frac{\partial M_i}{\partial y_j} \le 0, i,j = 1,2
			$$
			where $L= (1-x) - \frac{\alpha y}{1+h_1\alpha x} - \frac{\beta z}{1+h_2\beta x},$ $ M_1 = -u + \frac{e_1 \alpha x}{1+h_1\alpha x} - \frac{e_1 \alpha y}{1+h_1 \alpha x} -c_1z$, and $M_2 = -w + \frac{e_2 \alpha x}{1+h_2 \beta x} - \frac{e_2 \beta z}{1+h_2\beta  x} -c_2y$.
			
			\item In the absence of predators, the prey species x grows to carrying capacity, i.e.
			$$
			L(0,0,0)>0, \frac{\partial L}{\partial x}(x,y,z) =-1<0, \exists k>0 \ni J(k,0,0) =0. \text{ Here } k=1.
			$$
			
			\item There are no equilibrium points on the y or z coordinate axes and no equilibrium point in $yz-$plane.
			
			\item The predator $y$ and the predator $z$ can survive on the prey. This means there exist equilibrium points $\tilde{E}(\tilde{x}, \tilde{y},0)$ and $\hat{E}(\hat{x},0, \hat{z})$ such that $L, M_i$ evaluated at $\tilde{E}$ and $\hat{E}$  are zero such that $\tilde{x}, \tilde{y}, \hat{x}, \hat{z}>0$ and $\hat{x}, \tilde{x} <k.$
		\end{enumerate}
		
		We start by studying the persistence of prey. For this, define the state space to be $Z:= \{ (x,y,z)| x,y,z \ge 0\}$ and $Z_+ :=  \{ (x,y,z)| x,y,z > 0\}$. As similar to above problems, we rewrite model (\ref{two predator model}) in the form (\ref{continuous time}) as
		\begin{align*}
			\begin{split}
				& \frac{dY}{dt} = \mathcal{F}(x, Y) \\
				& \frac{dx}{dt} = \mathcal{A}(x,y,z) x \\
			\end{split}
		\end{align*}
	where 
	$$
	\mathcal{A} = (1-x) - \frac{\alpha y}{1+h_1\alpha x} - \frac{\beta z}{1+h_2\beta x}, Y = 
	\begin{bmatrix}
		y \\
		z
	\end{bmatrix},
	\mathcal{F} = \begin{bmatrix}
		-uy + \frac{e_1 \alpha xy}{1+h_1\alpha x} - \frac{e_1 \alpha y^2}{1+h_1 \alpha x} -c_1yz \\
		-wz + \frac{e_2 \alpha xz}{1+h_2 \beta x} - \frac{e_2 \beta z^2}{1+h_2\beta  x} -c_2yz 
	\end{bmatrix}.
	$$ 
 Define the extinction set to be $X_0^x = \{ (x,y,z) | x=0 \}.$ Then, the dynamics on $X_0^x$ is given by the subsystem
	\begin{align*}
		\begin{split}
			&  \frac{dy}{dt} =-uy +  - e_1 \alpha y^2 -c_1yz \equiv -y( u +   e_1 \alpha y + c_1z ) \\
			&  \frac{dz}{dt} =-wz + - e_2 \beta z^2 -c_2yz \equiv  -z ( w + - e_2 \beta z + c_2y )\\
		\end{split}
	\end{align*}
	It is not difficult to see, by assumption c),  that $\lim_{t \rightarrow \infty} (y,z) = (0,0)$ which implies that $\Omega(X_0^x) = \{ (0,0,0) \} \equiv \{M_1 \}$. The linearized system in $U$ can be written as 
	$$
	\frac{dU}{dt} = U \left[   (1-x) - \frac{\alpha y}{1+h_1\alpha x} - \frac{\beta z}{1+h_2\beta x}  \right]. 
	$$
	Clearly, its fundamental solution is $ P(t, x,y,z) = e^{t[(1-x) - \frac{\alpha y}{1+h_1\alpha x} - \frac{\beta z}{1+h_2\beta x}] }$. The Lyapunov exponent is
	\begin{align*}
		\begin{split}
			\lambda (x,y,z) & =  \limsup_{t \rightarrow \infty } \frac{1}{t} \ln |P(t,z)|  = (1-x) - \frac{\alpha y}{1+h_1\alpha x} - \frac{\beta z}{1+h_2\beta x}\\
		\lambda (M_1)	= & 1 >0.
		\end{split}
	\end{align*}
	Therefore, by a similar argument as in above problems we can choose $\epsilon_x >0$ such that $\liminf_{t \rightarrow \infty } |x(t)| > \epsilon_x, \forall x_0, y_0, z_0 \in Z\backslash X_0^x$. By definition, $x$ is uniformly strongly persistent.
	
	To study the persistence of first predator, that is, $y$, rewrite model (\ref{two predator model}) in the form (\ref{continuous time}) as
	\begin{align*}
		\begin{split}
			& \frac{dY}{dt} = \mathcal{F}(y, Y) \\
			& \frac{dy}{dt} = \mathcal{A}(x,y,z) y \\
		\end{split}
	\end{align*}
	$$
	\mathcal{A} = -u + \frac{e_1 \alpha x}{1+h_1\alpha x} - \frac{e_1 \alpha y}{1+h_1 \alpha x} -c_1z, 
	 Y = 
	\begin{bmatrix}
		x \\
		z
	\end{bmatrix},
	\mathcal{F} = \begin{bmatrix}
		x(1-x) - \frac{\alpha y}{1+h_1\alpha xy} - \frac{\beta xz}{1+h_2\beta x} \\
		-wz + \frac{e_2 \alpha xz}{1+h_2 \beta x} - \frac{e_2 \beta z^2}{1+h_2\beta  x} -c_2yz 
	\end{bmatrix}.
	$$ 
	Define the extinction set to be $X_0^y = \{ (x,y,z) | \epsilon_x/2 \le x \le B_x, y=0 \}.$ Then, the dynamics on $X_0^y$ is given by the subsystem
	\begin{align*}
		\begin{split}
			&\frac{dx}{dt} = x(1-x) - \frac{\beta xz}{1+h_2\beta x} \\
			&  \frac{dz}{dt} =-wz + \frac{e_2 \alpha xz}{1+h_2 \beta x} - \frac{e_2 \beta z^2}{1+h_2\beta  x} \\
		\end{split}
		\end{align*}
		With the assumptions made above ( see \cite{alebraheem2012persistence}), there exist no limit cycles on $xy-$plane. Therefore, $\Omega (X_0^y) = \{ (\hat{x}, 0, \hat{z}) \} \equiv \{ M_2\}$. Again, the linearized system in $U$ can be written as		
		$$
		\frac{dU}{dt} = U \left[  -u + \frac{e_1 \alpha x}{1+h_1\alpha x} -c_1z   \right]. 
		$$
		and its  fundamental solution is $ P(t, x,y,z) = e^{t[  -u + \frac{e_1 \alpha x}{1+h_1\alpha x} -c_1z   ] }$. Therefore, the positivity of the Lyapunov exponent, $	\lambda (M_2)	= -u + \frac{e_1 \alpha \hat{x}}{1+h_1\alpha \hat{x}} -c_1 \hat{z}$, implies that $ e_1 > \frac{ ( u+c_1\hat{z}  ) ( 1+h_1\alpha \hat{x} ) }{\alpha \hat{x}}.$ 	Therefore, by a similar argument as in above problems we can choose $\epsilon_y >0$ such that $\liminf_{t \rightarrow \infty } |y(t)| > \epsilon_x, \forall x_0, y_0, z_0 \in Z\backslash X_0^x$. By definition, $y$ is uniformly strongly persistent.
		
		By a similar analysis, one can show the persistence of second predator, that is, that there exists $\epsilon_z>0$ and hence that $z$ is uniformly strongly persistent when $ e_2> \frac{ ( w+c_2 \tilde{y}  ) ( 1+h_2\beta \tilde{x} ) }{ \beta \tilde{x}}.$ 
		
		To summarize the all above for system (\ref{two predator model}), we have the following theorem:
		
		\begin{theorem}
			Assume $ e_1 > \frac{ ( u+c_1\hat{z}  ) ( 1+h_1\alpha \hat{x} ) }{\alpha \hat{x}}$ and $ e_2> \frac{ ( w+c_2 \tilde{y}  ) ( 1+h_2\beta \tilde{x} ) }{ \beta \tilde{x}}$. Then, the system (\ref{two predator model}) is uniformly persistent, that is, there exists $\epsilon >0$ such that $\liminf_{ t \rightarrow \infty }\{ x(t), y(t), z(t)\} > \epsilon$ for all $x_0, y_0, z_0 > 0$.
			\end{theorem}
			\begin{proof}
				In the above analysis, finally choose $ \epsilon = \min . \{ \epsilon_x, \epsilon_y, \epsilon_z\}$ and the theorem holds.
			\end{proof}

		\section{Discussion}
		We referred to many examples from the literature and discussed the persistence of each system. Some of them were solved using the spectral radius approach, while a few were solved using the Lyapunov exponent. Although this work is directly inspired by \cite{salceanu2009lyapunov} and \cite{smith2011dynamical}, the author believes that this note provides useful examples in order to understand the process by which a given dynamical system, generated by ordinary differential equations or by maps, is shown to persist.
		
		\textbf{Acknowledgment:} The author would like to thank Dr. Paul Salceanu, Associate Professor in the Department of Mathematics at the University of Louisiana at Lafayette (ULL), whose lectures inspired the concept of this note. Special thanks are also due to Dr. Azmy S. Ackleh, my current PhD supervisor in the Department of Mathematics at ULL, who originally developed the model in \cite{ackleh2023discrete}, in joint work with Dr. Amy Veprauskas, also from the Department of Mathematics at ULL. The author further thanks Mr. Neerob Basak, a graduate student in the Department of Mathematics at ULL, for his contribution to completing Theorem \ref{theorem Q din} and for his constructive feedback. The author acknowledges the use of AI tools for grammar correction.

	\bibliographystyle{apalike} 
	\bibliography{Persistence-theory}

\end{document}